\documentclass[12pt,reqno]{amsart}
\usepackage{amsmath} 

\usepackage{amsthm}
\usepackage{amssymb}
\usepackage{graphics}
\usepackage{latexsym}

\usepackage{cases}

\numberwithin{equation}{section}
\newtheorem{thm}{Theorem}[section]
\newtheorem{prop}[thm]{Proposition}
\newtheorem{lem}[thm]{Lemma}
\newtheorem{cor}[thm]{Corollary}

\theoremstyle{remark}
\newtheorem{rem}[thm]{Remark}

\theoremstyle{assumption}

\newcommand{\BBB}{\mathbb}
\newcommand{\R}{{\BBB R}}

\newcommand{\N}{{\BBB N}}

\newcommand{\lec}{{\ \lesssim \ }}

\newcommand{\cross}{\times}

\newcommand{\be}{\beta}

\newcommand{\si}{\sigma}

\newcommand{\ta}{\tau}
\newcommand{\x}{\xi}
\newcommand{\y}{\eta}

\newcommand{\La}{\Lambda}

\newcommand{\Om}{\Omega}

\newcommand{\F}{\mathcal{F}}

\newcommand{\ti}{\widetilde}
\newcommand{\ha}{\widehat}

\newcommand{\hs}{\hspace}

\title[Operator splitting for dispersion-generalized Benjamin-Ono equations]{Operator splitting for dispersion-generalized Benjamin-Ono equations
}

\author[Takanobu Tokumasu]{Takanobu Tokumasu}
\address[]{Graduate School of Mathematics, Nagoya University,
Chikusa-ku, Nagoya, 464-8602, Japan}

\begin{document}

\begin{abstract}
We consider the operator splitting for a class of nonlinear equation, which includes the KdV equation, the Benjamin-Ono equation, and the Burgers equation. We prove a first-order approxomation in $\Delta t$ in the Sobolev space for the Godunov splitting, and second-order approximation for the Strang splitting.
\end{abstract}
\maketitle
\setcounter{page}{001}

\begin{section}{Introduction}
In this paper, we consider the operator splitting of
\begin{equation}
\begin{cases}
\partial_{t} u + u u_{x} - Ku = 0, \hs{15pt}  (x,t) \in \R \times [0,T],  \\
u(\cdot,0) = u_{0} \in H^{s} (\R)
\end{cases}
\label{eq:DBO}
\end{equation}
where $K = \F^{-1}[k(\x)\F]$ is a Fourier multiplier, and $u_{0}$ and $u$ are $\R$-valued. Throughout this paper, we assume that $p \in [0,\infty)$ and $k$ satisfies
\begin{align}
\Re(k(\x)) \leq 0 \, ,  \hs{15pt} k(-\x) = \overline{k(\x)} \, ,  \hs{15pt} | k(\x) | \lec \langle \x \rangle^{p} \, ,\label{K1} \\
|(\x + \y)k(\x + \y) - \y k(\y) - \x k(\x)| \lec |\x| \langle \y \rangle^{p} + |\y| \langle \x \rangle^{p} .
\label{K2}
\end{align}
for all $\x \, , \, \y \in \R$, where $\langle \x \rangle = (1 + |\x|^{2})^{1/2}$. Note that \eqref{K2} is satisfied if $k \in C^{1}(\R)$. Here we give some examples. When $k(\x) = - i \x^{3}$, \eqref{eq:DBO} is the Korteweg-de Vries (KdV) equation. When $k(\x) = i \x |\x|$, \eqref{eq:DBO} is the Benjamio-Ono equation. When $k(\x) = - \x^{2}$, \eqref{eq:DBO} is the Burgers equation.
The following are the case $p$ is a fractional number. The extended Whitham equation, which is a model for surface water waves (see \cite{Dutykh} and \cite{Saut}), is
\begin{equation*}
u_{t} + u u_{x} - \int_{\R} e^{i x \x} (1 + \be |\x|^{2})^{\frac{1}{2}} \Big{(} \frac{\tanh \x}{\x} \Big{)}^{\frac{1}{2}}  i \xi \ha{u}(\x) d \x = 0,
\end{equation*}
where $\be$ is a measure of surface tension. (see \cite{Saut}) When $\be > 0$, this satisfies \eqref{K1} and \eqref{K2} with $p=3/2$. The case $\be = 0$ is the case of no surface tension and called the Whitham equation (see \cite{Dutykh}), and also the model for purely gravitational waves (see \cite{Saut}), and this satisfies \eqref{K1} and \eqref{K2} with $p = 1/2$.

The operator splitting method is very popular for numerical computing, and this is explained formally as follows.
Let $\Phi_{C}[t_{0}]u_{0} \in X$, where $X$ is some normed space, denotes the solution of the differential equation $\partial_{t}u = C(u)$ at $t= t_{0}$, where $u_{0}$ is the initial value. Typical $C$ includes a differential operator in the spatial variables.
Here, we assume that $C = A + B$. In our setting, $A = K$ and $B(u) = - u u_{x}$. In the Godunov operator splitting, the approximate solution is
\begin{equation}
u(t_{n}) \approx \Big{(} \Phi_{A}(\Delta t) \circ \Phi_{B}(\Delta t) \Big{)}^{n} u_{0}, \nonumber
\end{equation}
where $n \in \N$, $\Delta t \ll 1$, and $t_{n} = n \Delta t$. In the Strang operator splitting, the approximate solution is
\begin{equation}
u(t_{n}) \approx \Big{(} \Phi_{A}(\frac{\Delta t}{2}) \circ \Phi_{B}(\Delta t)  \circ \Phi_{A}(\frac{\Delta t}{2}) \Big{)}^{n} u_{0}. \nonumber
\end{equation}

Let $N = N(\Delta t, T)$ be the largest integer such that $N \Delta t \leq T$ for given $T>0$ and $\Delta t > 0$. Let $n \in \N$ satisfy $1 \leq n \leq N$. The Godunov splitting for \eqref{eq:DBO} is defined in $(x,t,\ta) \in \R \cross \Pi_{G}^{(n)}$ as
\begin{numcases}
{}
\partial_{t} v + v v_{x} = 0,  \hs{15pt} x \in \R \, , \, ( t,\ta) \in \Sigma_{G,1}^{(n)},
\label{eq:B} \\
\partial_{\tau} v - Kv = 0, \hs{15pt} x \in \R \, , \, (t,\ta) \in \Sigma_{G,2}^{(n)},
\label{eq:L} \\
v (0,0) = u_{0} \in H^{s}(\R), \label{IVL}
\end{numcases}
where $\Pi_{G}^{(n)} = \Sigma_{G,1}^{(n)} \cup \Sigma_{G,2}^{(n)}$ and 
\begin{equation}
\Sigma_{G,1}^{(n)} = \bigcup_{l = 1}^{n} \Om_{1}^{(l)} \, , \, \Sigma_{G,2}^{(n)} = \bigcup_{l = 1}^{n} \Om_{2}^{(l)}, \nonumber
\end{equation}
and
\begin{equation}
\Om_{1}^{(n)} = ( t_{n-1} , t_{n} ] \cross \{ t_{n-1} \} \, , \, \Om_{2}^{(n)} = [ t_{n-1} , t_{n} ] \cross ( t_{n-1} ,t_{n} ]. \nonumber
\end{equation}
The Strang splitting for \eqref{eq:DBO} is defined in $(x,t,\tau) \in \R \cross \Pi_{S}^{(n)}$ as
\begin{numcases}
{}
\partial_{t} v + v v_{x} = 0, \hs{15pt} x \in \R \, , \, (t,\ta) \in \Sigma_{S,1}^{(n)},
\label{eq:SB} \\
\partial_{\tau} v - Kv = 0, \hs{15pt} x \in \R \, , \, (t,\ta) \in \Sigma_{S,2}^{(n)},
\label{eq:SL} \\
v(0,0) = u_{0} \in H^{s} (\R), \label{IVS}
\end{numcases}
where $\Pi_{S}^{(n)} = \Sigma_{S,1}^{(n)} \cup \Sigma_{S,2}^{(n)}$ and
\begin{equation}
\Sigma_{S,1}^{(n)} = \bigcup_{l = 1}^{n} \Big{(} \Om_{1,1}^{(l)} \cup \Om_{1,2}^{(l)} \Big{)} \, , \, \Sigma_{S,2}^{(n)} = \bigcup_{l = 1}^{n} \Big{(} \Om_{2,1}^{(l)} \cup \Om_{2,2}^{(l)} \Big{)}. \nonumber
\end{equation}
and
\begin{align}
\Om_{1,1}^{(n)} = ( t_{n-1} , t_{n- 1/2} ] \cross \{ t_{n-1} \} \, , \, \Om_{1,2}^{(n)} = ( t_{n- 1/2} , t_{n} ] \cross [ t_{n- 1/2} , t_{n} ] \, , \nonumber \\
\Om_{2,1}^{(n)} = [ t_{n-1} , t_{n- 1/2} ] \cross ( t_{n-1} ,t_{n- 1/2} ] \, , \, \Om_{2,2}^{(n)} = \{ t_{n- 1/2} \} \cross ( t_{n- 1/2} , t_{n} ], \nonumber
\end{align}
where $t_{n - 1/2} = (n - 1/2) \Delta t$ for $n \in \N$.

Since the splitting method requires the existence of the solution for the full equation (\ref{eq:DBO}), we give some known results. For the KdV equation, Colliander-Keel-Steffilani-Takaoka-Tao proved global-wellposedness (GWP) in $H^{s}$ for $s > -3/4$ (see \cite{Colliander}) and Kishimoto proved GWP in $H^{-3/4}$ (see \cite{Kishimoto}). For the Benjamin-Ono equation, Tao proved GWP in $H^{1}$ (see \cite{Tao2}) and Ionescu-Kenig proved  GWP in $L^{2}$ (see \cite{Ionescu}). For $k(\x) = -i \x |\x|^{p-1}$, Guo proved the local-wellposedness (LWP) in $H^{3-p}$ for $p \in [2,3]$ and the GWP in $H^{(p-1)/2}$ for $p \in (7/3, 3]$ are also proved in \cite{Guo}.

On the other hand, we have no results for the the global existence of the solution for the full equation (\ref{eq:DBO}) for general $k$. Throughout this paper, we assume that there exists a solution $u \in C([0,T]:H^{s})$ of (\ref{eq:DBO}) and a constant $C_{0} > 0$ satisfies
\begin{equation}
\sup_{t \in [0,T]} \| u(t) \|_{H^{s}(\R)} \leq C_{0}. \label{bdd-u}
\end{equation}

The splitting method also requires the solvability for (\ref{eq:B}), (\ref{eq:L}), (\ref{eq:SB}), and (\ref{eq:SL}). The global existence for the linear equations are obvious. In \cite{Kato} Kato proved the existence and the uniqueness of a local solution for a class of nonlinear equations, which includes the nonlinear equations, (\ref{eq:B}) and (\ref{eq:SB}). (Theorem \ref{thm:Burger})

In this paper, we prove two types of the error estimate. These proofs are based on the method used in \cite{Tao}. The first is the first-order approximation in $\Delta t$ for the Godunov splitting, that is the following.
\begin{thm} \label{thm:L}
Let $T > 0$ and $s > 3/2 + \max \{1,p \}$. Assume that $u \in C([0,T]:H^{s}(\R))$ satisfies $(\ref{eq:DBO})$ on $[0,T]$ and $C_{0}>0$ satisfies $(\ref{bdd-u})$. Then there exists $\overline{\Delta t}= \overline{\Delta t}(C_{0}, s, p, T) > 0$ such that for all $\Delta t \in [0,\overline{\Delta t}]$, there exists a unique solution $v \in C(\Pi_{G}^{(N)} : H^{s})$ of $(\ref{eq:B})$--$(\ref{IVL})$. In addition, there exists $C = C (C_{0}, s, p, T) >0 $ such that
\begin{equation}
\sup_{ t \in [0,N \Delta t] } \| v(t,t) - u(t) \|_{H^{s - \max \{ 1,p \}}} \leq C \Delta t. \nonumber
\end{equation}
\end{thm}
\begin{rem}\label{rem:rev}
Theorem \ref{thm:L} is also true for the Cauchy problem of
\begin{numcases}
{}
\partial_{t} v - Kv = 0,  \hs{15pt} x \in \R \, , \, (t,\ta) \in \Sigma_{G,1}^{(n)},
\nonumber \\
\partial_{\ta} v + v v_{x} = 0, \hs{15pt} x \in \R \, , \, (t,\ta) \in \Sigma_{G,2}^{(n)},
\nonumber \\
v (0,0) = u_{0} \in H^{s}(\R). \nonumber
\end{numcases}
The proof is the same as that of Theorem \ref{thm:L}.
\end{rem}
The second is the second-order approximation in $\Delta t$ for the Strang splitting, that is the following.
\begin{thm}\label{thm:S}
Let $T > 0$ and $s > 3/2 + 3 \max \{1,p \}$. Assume that $u \in C([0,T]:H^{s}(\R))$ satisfies $(\ref{eq:DBO})$ on $[0,T]$ and $C_{0} > 0$ satisfies $(\ref{bdd-u})$. Then there exists $\overline{\Delta t} = \overline{\Delta t}( C_{0}, s, p, T) > 0$ such that for all $\Delta t \in [0,\overline{\Delta t}]$, there exists a unique solution $v \in C(\Pi_{S}^{(N)} : H^{s})$ of the Cauchy problem $(\ref{eq:SB})$--$(\ref{IVS})$. In addition, there exists $C = C (C_{0}, s, p, T) >0 $ such that
\begin{equation}
\sup_{ t \in [0,N \Delta t] } \| v(t,t) - u(t) \|_{H^{s - 3 \max \{ 1,p \}}} \leq C (\Delta t)^{2}. \nonumber
\end{equation}
\end{thm}
\begin{rem}
Theorem \ref{thm:S} is also true for the Cauchy problem of
\begin{numcases}
{}
\partial_{t} v - Kv = 0, \hs{15pt} x \in \R \, , \, (t,\ta) \in \Sigma_{S,1}^{(n)},
\nonumber \\
\partial_{\ta} v + v v_{x} = 0, \hs{15pt} x \in \R \, , \, (t,\ta) \in \Sigma_{S,2}^{(n)},
\nonumber \\
v(0,0) = u_{0} \in H^{s} (\R). \nonumber
\end{numcases}
The proof is the same as that of Theorem \ref{thm:S}.
\end{rem}
Here we mention the privious results for the splitting method. 
Holden-Karlsen-Risebro-Tao proved the first-order approximation in $\Delta t$ for the KdV equation in $H^{s-3}$ when $u_{0} \in H^{s}$ and $s \geq 5$ is an odd integer in \cite{Tao}. In the same paper, they also proved the second-order approximation for the KdV equation in $H^{s-9}$ when $u_{0} \in H^{s}$ and $s \geq 17$ is an odd integer. Our result is a generalization of these results because we can apply Theorem \ref{thm:L} for the KdV equation for $s >9/2$ and Theorem \ref{thm:S} for the KdV equation for $s > 21/2$, respectively.  
Holden-Karlsen-Risebro proved the first-order approximation in $\Delta t$ in $H^{s-p}$ and the second-order approximation in $H^{s-2p+1}$ for the case $k$ is a polynomial, where $s$ is sufficiently large, $u_{0} \in H^{s}$, and $p \geq 2$ is the degree of $k$ (see \cite{Risebro}). 
Dutta-Holden-Koley-Risebro proved the first-order approximation in $\Delta t$ in $L^{2}$ for the Benjamin-Ono equation for $u_{0} \in H^{5/2}$ and the second-order approximation in $L^{2}$ for the Benjamin-Ono equation for $u_{0} \in H^{9/2}$ in \cite{Risebro-BO}. As far as the author know, no previous results of the splitting method exist for the extended Whitham equation mentioned above. On the other hand, our results include that for the extended Whitham equation, because our theorems work for $p \geq 0$.

The rest of the paper is organized as follows:
In Section 2, we mention some preliminary lemmas for Section 3 and 4. 
In section 3 and 4, we prove the main results for the Godunov and the Strang splittings, respectively.
\end{section}

\begin{section}{preliminaries}

First, we mention the solvability of the inviscid Burgers equation
\begin{numcases}
{}
\partial_{t} v + v v_{x} = 0, \hs{15pt}  x \in \R \, , \, t \in [0,T'],
 \label{QL} \\
v(0) = v_{0} \in H^{s} (\R).  \label{IV-QL}
\end{numcases}

\begin{thm}\label{thm:Burger}
Let $s > 3/2$. Then, there exists $T' = T'(s,\|v_{0}\|_{H^{s}(\R)})>0$ such that for all $T_{0} \leq T'$ there exists a unique solution $v \in C([0,T_{0}]:H^{s}(\R)) \cap C^{1}([0,T_{0}]:H^{s-1}(\R))$ of $(\ref{QL})$ and $(\ref{IV-QL})$ defined on $t \in [0,T_{0}]$.
\end{thm}
We have Theorem \ref{thm:Burger} from Theorem II in \cite{Kato} by Kato.

\begin{cor}\label{cor:sol}
Let $\si > 3/2$, $\ta_{0} \geq 0$, $v(\ta_{0}, \ta_{0}) \in H^{\si}(\R)$, and $M>0$ satisfy $\| v(\ta_{0},\ta_{0}) \|_{H^{\si}} \leq M$. Then, there exists $\overline{\Delta t}_{B} = \overline{\Delta t}_{B}(\si,M) > 0$ such that for all $\Delta t \leq \overline{\Delta t}_{B}$ there exists a unique solution $v \in C([\ta_{0},\ta_{0} + \Delta t]^{2}:H^{\si}(\R)) \cap C^{1}([\ta_{0},\ta_{0} + \Delta t]^{2}:H^{\si-1}(\R))$ of
\begin{numcases}
{}
\partial_{t} v + v v_{x} = 0,  \hs{15pt} x \in \R \, , \, ( t,\ta) \in (\ta_{0}, \ta_{0} + \Delta t] \times \{ \ta_{0} \},\label{solB}\\
\partial_{\tau} v - Kv = 0, \hs{15pt} x \in \R \, , \, (t,\ta) \in [\ta_{0},\ta_{0} + \Delta t] \times (\ta_{0}, \ta_{0} + \Delta t].\label{solL}
\end{numcases}
\end{cor}
We can solve (\ref{solL}) as $v(t,\ta,x) = \F_{\x}^{-1} [e^{k(\x) (\ta - \si)}\F_{x}[v(t,\si,x)]]$. Therefore we have Corollary \ref{cor:sol} from Theorem \ref{thm:Burger}.

\begin{lem}\label{lem:ibp2}
Let $s \geq \si > 3/2$. Then there exists $C = C(s)>0$ such that for all $f$, $g \in H^{s}(\R)$
\begin{align}
\| \partial_{x} \langle \partial_x \rangle^{s} (fg) - (\partial_{x} \langle \partial_x \rangle^{s} f)g - f(\partial_{x} \langle \partial_x \rangle^{s} g) \|_{L^{2}} \leq C (\| f \|_{H^{s}} \| g \|_{H^{\si}} + \| f \|_{H^{\si}} \| g \|_{H^{s}}),  \nonumber
\end{align}
where $\langle \partial_x \rangle^{s} = \F_{\x}^{-1} \langle \x \rangle^{s} \F_{x}$.
\end{lem}

\begin{proof}
We put $h(\x) = \x \langle \x \rangle^{s}$. First, we prove
\begin{align}
|h(\x) - h(\x - \x_{1}) - h(\x_{1})| \leq C (|\x_{1}|\langle \x - \x_{1} \rangle^{s} + |\x - \x_{1}| \langle \x_{1} \rangle^{s}).\label{eq:ibp21}
\end{align}
By symmetry, we only need to prove the case $|\x_{1}| \leq |\x - \x_{1}|$. We have $|h'(\x)| = |1 + (s+1) \x^{2}|\langle \x \rangle^{s - 2} \leq C \langle \x \rangle^{s}$. By the mean-value theorem, we have 
\begin{align}
|h(\x) - h(\x - \x_{1})| &= |\xi_{1}| |h'(\x - \theta \x_{1})| \leq C |\xi_{1}| \langle \x - \theta \x_{1} \rangle^{s} \leq C |\x_{1}| \langle \x - \x_{1} \rangle^{s}, \nonumber
\end{align}
where $\theta \in (0,1)$. Since $|\x_{1}| \leq |\x - \x_{1}|$, we have $|h(\x_{1})| \leq |\x_{1}| \langle \x - \x_{1} \rangle^{s}$. Thus we have (\ref{eq:ibp21}).
Next, we prove Lemma \ref{lem:ibp2}. By the Sobolev inequality, the Plancherel equality, and (\ref{eq:ibp21}), we have
\begin{align}
L.H.S. &\leq C \Big{\|} \int_{\R} \{ h(\x) - h(\x - \x_{1}) - h(\x_{1}) \} \ha{f}(\x - \x_{1}) \ha{g} (\x_{1}) d \x_{1} \Big{\|}_{L^{2}} \nonumber \\
&\leq C \Big{\|} \int_{\R} \{ |\x_{1}| \langle \x - \x_{1} \rangle^{s} + |\x - \x_{1}| \langle \x_{1} \rangle^{s} \} |\ha{f}(\x - \x_{1})| |\ha{g} (\x_{1})|  d \x_{1} \Big{\|}_{L^{2}} \nonumber \\
&\leq C ( \| \langle \x \rangle^{s} |\ha{f}| \|_{L^{2}} \| |\x| |\ha{g}| \|_{L^{1}} + \| \langle \x \rangle^{s} |\ha{g}| \|_{L^{2}} \| |\x| |\ha{f}| \|_{L^{1}}) \nonumber \\
&\leq C ( \| f \|_{H^{s}} \| g \|_{H^{\si}} + \| g \|_{H^{s}} \| f \|_{H^{\si}} ). \nonumber
\end{align}
\end{proof}
\begin{prop}\label{prop:ibp}
Let $s \geq \si >3/2$. Assume that $f$ and $g$ are $\R$-valued functions.\\
$(A)$ Then there exists $C = C(s)>0$ such that for all $f$, $g \in H^{s}(\R)$
\begin{equation}
| \langle f,(fg)_{x} \rangle_{H^{s}} | \leq C \| f \|_{H^{s}}^{2} \| g \|_{H^{s+1}}.
\label{eq:lemA1}
\end{equation}
$(B)$ Then there exists $C = C(s,\si)>0$ such that for all $f \in H^{s}(\R)$
\begin{equation}
| \langle f,f f_{x} \rangle_{H^{s}} | \leq C \| f \|_{H^{s}}^{2} \| f \|_{H^{\si}}.
\label{eq:lemA2}
\end{equation}
\end{prop}
\begin{proof}
First, we show (A). By the definition of inner product of $H^{s}(\R)$, we have
\begin{align}
\langle f, (fg)_{x} \rangle_{H^{s}} &= \langle \langle \partial_{x} \rangle^{s} f, (\partial_{x} \langle \partial_{x} \rangle^{s} f) g \rangle_{L^{2}} + \langle \langle \partial_{x} \rangle^{s} f, f (\partial_{x} \langle \partial_{x} \rangle^{s} g) \rangle_{L^{2}} \nonumber \\
&\quad + \langle \langle \partial_x \rangle^{s} f , \partial_{x} \langle \partial_x \rangle^{s} (fg) - (\partial_{x} \langle \partial_x \rangle^{s} f)g - f(\partial_{x} \langle \partial_x \rangle^{s} g) \rangle_{L^{2}}.\label{eq:ibp01}
\end{align}
For the first term, by integration by parts and the Sobolev inequality, we have
\begin{align}
| \langle \langle \partial_{x} \rangle^{s} f, (\partial_{x} \langle \partial_{x} \rangle^{s} f) g \rangle_{L^{2}} | &= | - \frac{1}{2} \langle \langle \partial_{x} \rangle^{s} f, (\langle \partial_{x} \rangle^{s} f) (\partial_{x} g) \rangle_{L^{2}} | \nonumber \\
&\leq C \| f \|_{H^{s}}^{2} \| g \|_{H^{\si}}.\label{eq:ibp02}
\end{align}
For the second term, by the Sobolev inequality, we have $| \langle \langle \partial_{x} \rangle^{s} f, f (\partial_{x} \langle \partial_{x} \rangle^{s} g) \rangle_{L^{2}} | \leq C \| f \|_{H^{s}} \| f \|_{H^{\si - 1}} \| g \|_{H^{s+1}}$. By Lemma \ref{lem:ibp2}, the third term in (\ref{eq:ibp01}) is bounded by
\begin{align}
&\| f \|_{H^{s}} \| \partial_{x} \langle \partial_x \rangle^{s} (fg) - (\partial_{x} \langle \partial_x \rangle^{s} f)g - f(\partial_{x} \langle \partial_x \rangle^{s} g) \|_{L^{2}} \nonumber \\ 
&\quad \leq C \| f \|_{H^{s}} (\| f \|_{H^{s}} \| g \|_{H^{\si}} + \| f \|_{H^{\si}} \| g \|_{H^{s}}).\label{eq:ibp03}
\end{align}
Since $s \geq \si$, we have the desired result. Next, we prove (B). We put $g = f$ in  (\ref{eq:ibp01}). Then the second term in (\ref{eq:ibp01}) is equal to the first term. Therefore, we have the desired result by (\ref{eq:ibp02}) and (\ref{eq:ibp03}).
\end{proof}


The following lemma is so called bootstrap lemma, which follows from the continuity of $v$ and the connectivity of $\Pi^{(n)}$.
\begin{lem}{(bootstrap lemma)}\label{lem:bootstrap}
Let $\si > 3/2$, $\Delta t > 0$, $n \in \N$, and $\Pi^{(n)} = \Pi_{G}^{(n)}$ or $\Pi_{S}^{(n)}$. Assume that $v \in C(\Pi^{(n)} : H^{\si})$ and $C_{1} > 0$ satisfy the following two conditions.\\
$(A)$ $\| v(0,0) \|_{H^{\si}} \leq C_{1}$. \\
$(B)$ $\sup_{(t,\ta) \in \Pi^{(n)}} \| v(t,\ta) \|_{H^{\si}} \leq C_{1} / 2$ holds if $\sup_{(t,\ta) \in \Pi^{(n)}} \| v(t,\ta) \|_{H^{\si}} \leq C_{1}$. \\
Then we have
\begin{equation}
\sup_{(t,\ta) \in \Pi^{(n)}} \| v(t,\ta) \|_{H^{\si}} \leq C_{1} / 2. \nonumber
\end{equation}
\end{lem}

\begin{lem}\label{lem:LHs}
Let $\Delta t>0$, $T>0$, $n \in \N$ such that $n \Delta t \leq T$, $\Pi^{(n)} = \Pi_{G}^{(n)}$ (resp. $\Pi_{S}^{(n)}$) and $s \geq \si > 3/2$. Assume that $v \in C(\Pi^{(n)}: H^{\si})$ satisfy $(\ref{eq:B})$--$(\ref{IVL})$ (resp. $(\ref{eq:SB})$--$(\ref{IVS})$) on $\Pi^{(n)}$. Assume that $C_{1} > 0$ and $v$ satisfy $\sup_{(t,\ta) \in \Pi^{(n)}} \| v(t,\ta) \|_{H^{\si}} \leq C_{1}$. Then there exists $C_{1}' = C_{1}' (\| u_{0} \|_{H^{s}}, C_{1}, s, \si, T) > 0$ such that
\begin{equation}
\sup_{(t,\tau) \in \Pi^{(n)} } \| v(t,\tau) \|_{H^{s}} \leq C_{1}'. \label{eq:LHs}
\end{equation}
\end{lem}

\begin{proof}
We prove the case $\Pi^{(n)} = \Pi_{G}^{(n)}$. From (\ref{K1}) and (\ref{eq:L}), $\| v(t,\ta) \|_{H^{s}}$ is monotonically decresing with respect to $\ta$ in $\Om_{2}^{(n)}$. So we only need to prove the boundness of $\| v(t,t_{n-1}) \|_{H^{s}}$ in $\Om_{1}^{(n)}$. By (\ref{eq:B}), (\ref{eq:lemA2}), and $\sup_{(t,\ta) \in \Pi_{G}^{(n)}} \| v(t,\ta) \|_{H^{\si}} \leq C_{1}$, we have
\begin{align}
\frac{d}{dt} \| v(t, t_{n-1}) \|_{H^{s}}^{2} &= - 2 \langle v(t,t_{n-1}) , v(t,t_{n-1}) v_{x}(t,t_{n-1}) \rangle_{H^{s}} \nonumber \\
&\leq C \| v(t,t_{n-1}) \|_{H^{s}}^{2} \| v(t,t_{n-1}) \|_{H^{\si}} \nonumber \\
&\leq C C_{1} \| v(t,t_{n-1}) \|_{H^{s}}^{2}. \label{eq:LHs2}
\end{align}
We have $\| v(t,t_{n-1}) \|_{H^{s}} \leq \| v(t_{n-1},t_{n-1}) \|_{H^{s}}e^{C C_{1} (t-t_{n-1})}$ by applying the Gronwall inequality to (\ref{eq:LHs2}). Therefore it follows that
\begin{align}
\| v(t,\ta) \|_{H^{s}} &\leq \| v(t_{n-1},t_{n-1}) \|_{H^{s}} e^{C C_{1} (t-t_{n-1})} \nonumber \\
&\leq \| v(t_{n-2},t_{n-2}) \|_{H^{s}} e^{C C_{1} (t_{n-1}-t_{n-2})} e^{C C_{1} (t-t_{n-1})} \nonumber \\
&\leq \cdots \leq \| v(0,0) \|_{H^{s}} e^{C C_{1} t} \leq \| u_{0} \|_{H^{s}} e^{C C_{1} T}. 
\end{align}
Similar arguments apply to the case $\Pi^{(n)} = \Pi_{S}^{(n)}$.
\end{proof}
\end{section}

\begin{section}{Estimate for the Godunov splitting}

The main estimate in this section is Proposition \ref{prop:bootL1} below.
\begin{prop}\label{prop:bootL1}
Let $\Delta t>0$, $T>0$, $n \in \N$ such that $n \Delta t \leq T$, and $s_{1} = s - \max \{1,p \} > 3/2$. Assume that $v \in C(\Pi_{G}^{(n)} : H^{s})$  and $C_{1}'>0$ satisfy $(\ref{eq:B})$--$(\ref{IVL})$ and $(\ref{eq:LHs})$. Then there exists $C_{2} = C_{2}(C_{0}, C_{1}', s, s_{1}, T)>0$ such that
\begin{equation}
\sup_{t \in [0, n \Delta t]} \| v(t,t) - u(t) \|_{H^{s_{1}}} \leq C_{2} \Delta t. \nonumber
\end{equation}
\end{prop}
Proposition \ref{prop:bootL1} follows from Lemmas \ref{lem:bGWL1} and \ref{lem:bGWL2} below.
\begin{lem}\label{lem:bGWL1}
Let $F = v_{t} + v v_{x}$ and $F(t) = F(t,t)$. Under the same assumptions of Proposition \ref{prop:bootL1}, there exists $C = C(C_{0}, C_{1}' ,s , s_{1}, T)>0$ such that
\begin{equation}
\| v(t,t) - u(t) \|_{H^{s_{1}}} \leq C \int_{0}^{t} \| F(t') \|_{H^{s_{1}}}  d t'
\label{eq:GronL}
\end{equation}
for all $t \in [0,n \Delta t]$.
\end{lem}
\begin{proof}
Let $w(t) = v(t,t) - u(t)$. By (\ref{eq:DBO}), (\ref{eq:B}), and the definition of $F$, we have
\begin{align}
\frac{\partial}{\partial t}w(t,x) &= v_{t} (t,\ta,x) |_{\ta = t} + v_{\ta} (t,\ta,x) |_{\ta = t} - u_{t} (t,x) \nonumber \\
&= - ww_{x} -(uw)_{x} + Kw + v_{t} + v v_{x} \nonumber \\
&= - ww_{x} -(uw)_{x} + Kw + F \label{eq:w'L}
\end{align}
In view of \eqref{eq:w'L}, we call $F$ the forcing term. Then we have
\begin{equation}
\frac{d}{dt} \| w(t) \|_{H^{s_{1}}}^{2} = 2 \langle w, - ww_{x} -(uw)_{x} + Kw + F \rangle_{H^{s_{1}}}.
\label{eq:L1}
\end{equation}
Note that $\langle w,Kw \rangle_{H^{s}} \leq 0$ from (\ref{K1}). By the Schwarz inequality, Proposition \ref{prop:ibp}, Lemma \ref{lem:LHs}, and (\ref{bdd-u}), we get
\begin{align}
\frac{d}{dt} \| w(t) \|_{H^{s_{1}}} &\leq C \{ \| w \|_{H^{s_{1}}} ( \| w \|_{H^{s_{1}}} + \| u \|_{H^{s_{1} + 1}} ) + \| F \|_{H^{s_{1}}} \} \nonumber \\
&\leq C (\| w \|_{H^{s_{1}}} + \| F \|_{H^{s_{1}}}). \label{eq:L1-1}
\end{align}
Here we used $\| w \|_{H^{s_{1}}} \leq \| u \|_{H^{s}} + \| v \|_{H^{s}} \leq C_{0} + C_{1}'$ and $\| u \|_{H^{s_{1} + 1}} \leq C_{0}$.  Applying the Gronwall inequality and $w(0) = 0$ to (\ref{eq:L1-1}), and we have
\begin{equation}
\|w(t)\|_{H^{s_{1}}} \leq C \int_{0}^{t} \| F(\si) \|_{H^{s_{1}}} d \si. \nonumber
\end{equation}
\end{proof}

\begin{lem}\label{lem:bGWL2}
Let $F = v_{t} + v v_{x}$ and $F(t) = F(t,t)$. Under the same assumptions of Proposition \ref{prop:bootL1}, there exists $C = C(\|u_{0}\|_{H^{s}}, C_{1}' , s, s_{1}, T)>0$ such that
\begin{equation}
\sup_{t \in [0,n \Delta t]} \| F(t) \|_{H^{s_{1}}} \leq C \Delta t.
\label{eq:FL1}
\end{equation}
\end{lem}
\begin{proof}
By (\ref{eq:L}) and the definition of $F$, the forcing term $F$ satisfies
\begin{equation}
\partial_{\tau} F(t,\ta) - KF(t,\ta) = X(t,\ta) \label{eq:L2}
\end{equation}
in $(t,\ta) \in \Pi_{G}^{(n)}$, where $X(t,\ta) = X_{1}(t,\ta) + X_{2}(t,\ta)$ and $X_{1}(t,\ta)$ and $X_{1}(t,\ta)$ are defined as below.
\begin{equation}
X_{1} (t,\ta) =  - \Big{\{} \frac{1}{2} K(v^{2})_{x} - v K v_{x} \Big{\}} \, , \, X_{2} (t,\ta) = - v_{x} K v  \nonumber
\end{equation}
By (\ref{eq:L2}), we get
\begin{equation}
\partial_{\ta} \| F(t,\ta) \|_{H^{s}}^{2} = 2 \langle F , KF \rangle_{H^{s}} + 2 \langle F,X \rangle_{H^{s}}.
\label{eq:L3}
\end{equation}
The first term in (\ref{eq:L3}) is equal to or less than 0 because of (\ref{K1}).
For the second term, by (\ref{K2}), the Sovolev inequality, and Lemma \ref{lem:LHs}, we have
\begin{align}
\| X_{1} \|_{H^{s_{1}}} &= \frac{1}{2} \| \langle \xi \rangle^{s_{1}} \int_{\R} \{ \x k(\x) - \x_{1} k(\x_{1}) - (\x - \x_{1}) k(\x - \x_{1}) \} \ha{v}(t,\ta,\x - \x_{1}) \ha{v}(t,\ta,\x_{1}) d \x_{1} \|_{L^{2}} \nonumber \\
&\quad \leq C ( \| \langle \x \rangle^{s_{1}} |\x| |\ha{v}| \|_{L^{2}} \| \langle \x \rangle^{p} |\ha{v}| \|_{L^{1}} + \| \langle \x \rangle^{s_{1}} \langle \x \rangle^{p} | \ha{v}| \|_{L^{2}} \| |\x| |\ha{v}| \|_{L^{1}})  \nonumber \\
&\quad \leq C ( \| v \|_{H^{s_{1}+1}} \| v \|_{H^{s_{1} + p}} + \| v \|_{H^{s_{1}+p}} \| v \|_{H^{s_{1} + 1}} ) \leq C, \label{eq:L4+}
\end{align}
where we used $s = s_{1} + \max \{ 1,p \}$ and $s_{1} > 3/2$ to have the last inequality in $(\ref{eq:L4+})$. Under the same assumption of  $s$ and $s_{1}$ written above, by the Sobolev inequality, we have $\| X_{2} \|_{H^{s_{1}}} \leq C \| v \|_{H^{s_{1} + 1}} \| v \|_{H^{s_{1} + p}} \leq C$. Therefore, we have 
\begin{equation}
\partial_{\ta} \| F(t,\ta) \|_{H^{s}} \leq C.\label{eq:L6}
\end{equation}
Take $k \in \N$ such that $(t,\ta) \in [t_{k},t_{k+1}]^{2}$. Since $F(t,t_{k}) = 0$,
\begin{align}
\| F(t,\ta) \|_{H^{s}} &= \| F(t,\ta) \|_{H^{s}} - \| F(t,t_{k}) \|_{H^{s}} \nonumber \\
&= \int_{t_{k}}^{\ta} \partial_{\si} \| F(t,\si) \|_{H^{s}} d \si \leq C|\ta - t_{k}| \leq C \Delta t. \nonumber
\end{align}
\end{proof}
As a corollary of Proposition \ref{prop:bootL1}, we have the following.
\begin{cor}\label{cor:rmL}
Let $T>0$ and $s_{1} = s - \max \{1,p \} > 3/2$. Then there exist $\overline{\Delta t}_{*} = \overline{\Delta t}_{*}(C_{0}, s, s_{1}, T) > 0$ and $C_{*} = C_{*}(C_{0},s,s_{1},T)>0$ such that, for all $\Delta t \leq \overline{\Delta t}_{*}$, $n \in \N$ satisfying $n \Delta t \leq T$, and $v \in C(\Pi_{G}^{(n)} : H^{s_{1}})$ satisfying $(\ref{eq:B})$--$(\ref{IVL})$ on $\Pi_{G}^{(n)}$, it follows that $v \in C(\Pi_{G}^{(n)} : H^{s})$ and
\begin{align}
\sup_{(t,\ta) \in \Pi_{G}^{(n)}} \| v(t,\ta) \|_{H^{s_{1}}} &\leq 2 C_{0}, \label{eq:Lgoal} \\
\sup_{t \in [0, n \Delta t]} \| v(t,t) - u(t) \|_{H^{s_{1}}} &\leq C_{*} \Delta t. \label{eq:vtt-u}
\end{align}
\end{cor}
\begin{proof}
First, we prove \eqref{eq:Lgoal}. For that purpose, we only need to prove (A) and (B) of Lemma \ref{lem:bootstrap} with $C_{1} = 4 C_{0}$.  Obviously, (A) holds by \eqref{bdd-u}. Next, we prove (B). Assume that $\sup_{(t,\ta) \in \Pi_{G}^{(n)}} \| v(t,\ta) \|_{H^{s_{1}}} \leq 4 C_{0}$. By Lemma \ref{lem:LHs} with $C_{1} = 4 C_{0}$, there exists $C = C( C_{0}, s, s_{1}, T) >0$ such that $\sup_{(t,\ta) \in \Pi_{G}^{(n)}} \| v(t,\ta) \|_{H^{s}} \leq C$. By $s_{1} + p \leq s$, (\ref{K1}), and \eqref{eq:L}, we have
\begin{align}
\|v(t,\ta) - v(t,t)\|_{H^{s_{1}}} &\leq \int_{t}^{\ta} \|\partial_{\sigma} v(t,\sigma)\|_{H^{s_{1}}} d \sigma \nonumber \\
&\leq |t - \ta| \sup_{(t,\ta) \in \Pi_{G}^{(n)}} \|v(t,\ta)\|_{H^{s_{1} + p}} \leq C \Delta t. \label{eq:vtt-vtta}
\end{align}
Then, by Proposition \ref{prop:bootL1}, it follows that for all $\Delta t \leq \overline{\Delta t}_{*}$,
\begin{align}
\| v(t,\ta) \|_{H^{3/2 + \epsilon}} &\leq \| v(t,\tau) - v(t,t) \|_{H^{s_{1}}} + \| v(t,t) - u(t) \|_{H^{s_{1}}} + \| u(t) \| _{H^{s}} \nonumber \\
&\leq C \Delta t + C_{0} \leq 2 C_{0}. \nonumber
\end{align}
Here we take $\overline{\Delta t}_{*}$ such that $C \overline{\Delta t}_{*} = C_{0}$. Thus, we obtain (B). Therefore, we have \eqref{eq:Lgoal} by Lemma \ref{lem:bootstrap}.\\
By applying Lemma \ref{lem:LHs} to \eqref{eq:Lgoal}, it is also proved that $v \in C(\Pi_{G}^{(n)} : H^{s})$ and there exists $C = C(C_{0}, s, s_{1}, T) >0$ such that $\sup_{(t,\ta) \in \Pi_{G}^{(n)}} \| v(t,\ta) \|_{H^{s}} \leq C$. Therefore, by Proposition \ref{prop:bootL1}, we have (\ref{eq:vtt-u}).
\end{proof}

Finally we prove Theorem \ref{thm:L}.
\begin{proof}
Let $\overline{\Delta t} = \min \{ \overline{\Delta t}_{*} , \overline{\Delta t}_{B} (s_{1}, 2C_{0}) \}$, where $\overline{\Delta t}_{B} = \overline{\Delta t}_{B}(\si,M)$ is defined in Corollary \ref{cor:sol}. Note that $\overline{\Delta t} = \overline{\Delta t}(C_{0}, s, s_{1}, T)$.

We put conditions $(A)_{n}$ and $(B)_{n}$ for $n \in \N$ satisfying $1 \leq n \leq N$ as below.\\
$(A)_{n}$ : For any $\Delta t \leq \overline{\Delta t}$, there exists a unique solution $v \in C(\Pi_{G}^{(n)}:H^{s_{1}}(\R))$ which satisfies $(\ref{eq:B})$--$(\ref{IVL})$. \\
$(B)_{n}$ : For any $\Delta t \leq \overline{\Delta t}$, the solution $v \in C(\Pi_{G}^{(n)}:H^{s_{1}}(\R))$ of \eqref{eq:B}--\eqref{IVL} satisfies $v \in C(\Pi_{G}^{(n)}:H^{s}(\R))$, \eqref{eq:Lgoal}, and \eqref{eq:vtt-u}.

The proof is by induction on $n$. Obviously, $(A)_{1}$ and $(B)_{1}$ are true. 
Let $l \in \N$ such that $1 \leq l \leq N-1$. We assume $(A)_{l}$ and $(B)_{l}$, and prove $(A)_{l+1}$ and $(B)_{l+1}$. First, we prove that $(A)_{l+1}$ holds. Since $\Delta t \leq \overline{\Delta t} \leq \overline{\Delta t}_{B*}$, we have $(A)_{l+1}$ by \eqref{eq:Lgoal} with $n = l$ and Corollary \ref{cor:sol}. Next, we prove $(B)_{l+1}$ from $(A)_{l+1}$, but this has already been proved as Corollary \ref{cor:rmL}. 

By induction on $n$, we have Theorem \ref{thm:L} for $C = C_{2}(C_{0}, C_{1}'|_{C_{1} = 2C_{0}}, s, s_{1}, T)$, where $C_{2}$ is the constant in Proposition \ref{prop:bootL1}.
\end{proof}
\end{section}

\begin{section}{Estimate for the Strang splitting}

In this section, we prove Theorem \ref{thm:S}. We put $w(t) = v(t,t) - u(t)$, $\La_{1}^{(n)} = \cup_{l = 1}^{n}(\Om_{1,1}^{(l)} \cup \Om_{2,1}^{(l)})$, $\La_{2}^{(n)} = \cup_{l = 1}^{n}(\Om_{2,2}^{(l)} \cup \Om_{1,2}^{(l)})$, and tilde is a time-shift operator, that is $\ti{f}(t,x) = f (t + \Delta t / 2, x)$.

The main proposition in this section is Proposition \ref{prop:2order} below. We obtain Theorem \ref{thm:S} in the same manner as in Section 3 if we use Proposition \ref{prop:2order} instead of Proposition \ref{prop:bootL1}. Therefore, we only need to prove Proposition \ref{prop:2order}.
\begin{prop}\label{prop:2order}
Let $\Delta t>0$, $T>0$, $n \in \N$ such that $n \Delta t \leq T$, and $s_{2} = s - 3 \max \{1,p \} > 3/2$. Assume that there exists a unique solution $v \in C(\Pi_{S}^{(n)} : H^{s})$ of $(\ref{eq:SB})$--$(\ref{IVS})$ and a constant $C_{1}'>0$ satisfies $(\ref{eq:LHs})$. Then, there exists $C_{2} = C_{2}(C_{0}, C_{1}', s, s_{2}, T)>0$ such that
\begin{equation}
\sup_{t \in [0, t_{n} ]} \| w(t) \|_{H^{s_{2}}} \leq C_{2} (\Delta t)^{2}.
\label{eq:goalS}
\end{equation}
\end{prop}
For the proof of Proposition \ref{prop:2order}, we only need to prove
\begin{align}
\sup_{t \in [0, t_{n-1/2} ]} \| w(t) + \ti{w}(t) \|_{H^{s_{2}}} &\leq C' (\Delta t)^{2} \label{eq:ave} \\
\sup_{t \in [0, t_{n-1/2}]} \| w(t) - \ti{w}(t) \|_{H^{s_{2}}} &\leq C' (\Delta t)^{2} \label{eq:dif}
\end{align}
instead of (\ref{eq:goalS}), since
\begin{align}
\sup_{t \in [0, t_{n}]} \| w(t) \|_{H^{s_{2}}} &\leq \sup_{t \in [0, t_{n-1/2}]} \| w(t) \|_{H^{s_{2}}} + \sup_{t \in [0, t_{n-1/2}]} \| \ti{w}(t) \|_{H^{s_{2}}} \nonumber \\
&\leq \sup_{t \in [0, t_{n-1/2}]} \| w(t) + \ti{w}(t) \|_{H^{s_{2}}} + \sup_{t \in [0, t_{n-1/2}]} \| w(t) - \ti{w}(t) \|_{H^{s_{2}}}. \nonumber
\end{align}

First, we prepare some notations to prove (\ref{eq:ave}). We put
\begin{align}
F (t,\tau) &=
\begin{cases}
v_{t}(t,\tau) + v(t,\tau) \, v_{x}(t,\tau), \hs{12pt} (t,\tau) \in \La_{1}^{(n)}, \\
0, \hs{124pt} (t,\tau) \in \La_{2}^{(n)},
\end{cases}
\label{eq:ftf}\\
G(t,\tau) &=
\begin{cases}
0, \hs{124pt} (t,\tau) \in \La_{1}^{(n)},\\
v_{\tau}(t,\tau) - Kv(t,\tau), \hs{37pt} (t,\tau) \in \La_{2}^{(n)}. 
\end{cases}
\label{eq:ftg}
\end{align}
We also define the total forcing term $H$ in $\Pi_{S}^{(n)}$ as $H(t,\tau) = F(t,\tau) + G(t,\tau)$ and $H(t) = H(t,t)$.
By $(\ref{eq:DBO})$, $(\ref{eq:SL})$, and $(\ref{eq:ftf})$ for the case  $(t,t) \in \La_{1}^{(n)}$ and $(\ref{eq:DBO})$, $(\ref{eq:SB})$, and $(\ref{eq:ftg})$ for the case $(t,t) \in \La_{2}^{(n)}$, $w'$ is written in $t \in [0, t_{n}]$ as
\begin{equation}
w' = - ww_{x} - (uw)_{x} + Kw + H(t). \label{eq:w'}
\end{equation}
For simplicity, we put $z(t) = w(t) + \ti{w}(t)$ in $t \in [0, t_{n-1/2}]$.

Next, we prepare some lemmas to estimate $\| z(t) \|_{H^{s_{2}}}$.
\begin{lem}\label{lem:GWS}
Let $H(t,\ta) = F(t,\ta) + G(t,\ta)$ and $H(t)= H(t,t)$, where $F$ satisfies \eqref{eq:ftf} and $G$ satisfies \eqref{eq:ftg}. Under the same assumption of Proposition \ref{prop:2order}, there exists $C = C (C_{0}, C_{1}' ,s , s_{2}, T)>0$ such that for all $t \in [0, t_{n-1/2}]$,
\begin{align}
&\| z(t) \|_{H^{s_{2}}} \nonumber \\
&\quad \leq \| z(0) \|_{H^{s_{2}}} e^{C t} + C t \{ \sup_{t \in [0, t_{n-1/2}]} \| H(t) + \ti{H}(t) \|_{H^{s_{2}}} \nonumber \\
&\qquad + (\sup_{t \in [0, t_{n}]} \| w \|_{H^{s_{2} + 1}})^{2} + \sup_{t \in [0, t_{n}]} \| w \|_{H^{s_{2} + 1}} \sup_{t \in [0, t_{n-1/2}]} \| \ti{u} - u \|_{H^{s_{2} + 1}} \}.
\label{eq:z}
\end{align}
\end{lem}

\begin{proof}
By $\eqref{eq:w'}$, it follows that for $t \in [0, t_{n-1/2}]$,
\begin{equation}
\ti{w}' = \ti{H}(t) + K \ti{w} - (\ti{u} \ti{w})_{x} - \ti{w} \ti{w}_{x}. \label{eq:wt'}
\end{equation}
By $(\ref{eq:w'})$ and $(\ref{eq:wt'})$, we have the equation for $z'$ in $t \in [0, t_{n-1/2}]$ as below.
\begin{equation}
z' = H(t) + \ti{H}(t) - \Big{(} \frac{1}{2}z^{2} + uz \Big{)}_{x} + Kz - \Big{\{} \ti{w} (\ti{u} - u) + w \ti{w} \Big{\}}_{x}.
\label{eq:z2}
\end{equation}
Then it follows that
\begin{align}
\frac{d}{dt} \| z(t) \|_{H^{s_{2}}}^{2} = \Big{\langle} z , H(t) + \ti{H}(t) - \Big{(} \frac{1}{2}z^{2} + uz \Big{)}_{x} + Kz - \Big{\{} \ti{w} (\ti{u} - u) + w \ti{w} \Big{\}}_{x} \Big{\rangle}_{H^{s_{2}}}. \nonumber
\end{align}
Note that $\langle z, Kz \rangle_{H^{s_{2}}} \leq 0$ from \eqref{K1}. By the Sobolev inequality, Proposition \ref{prop:ibp}, Lemma \ref{lem:LHs}, and \eqref{bdd-u}, it follows that for $t \in [0, t_{n-1/2}]$,
\begin{align}
\frac{d}{dt} \| z(t) \|_{H^{s_{2}}} &\leq C \{ \| z \|_{H^{s_{2}}} + \| H(t) + \ti{H}(t) \|_{H^{s_{2}}} \nonumber \\
&\quad + \| w \|_{H^{s_{2} + 1}} \| \ti{w} \|_{H^{s_{2} + 1}} + \| \ti{w} \|_{H^{s_{2} + 1}} \| \ti{u} - u \|_{H^{s_{2} + 1}} \}.
\label{eq:z3}
\end{align}
Here we used the following inequality. 
\begin{align}
\| z \|_{H^{s_{2}+1}} &\leq \| w \|_{H^{s_{2}+1}} + \| \ti{w} \|_{H^{s_{2}+1}} \nonumber \\
&\leq \| u \|_{H^{s_{2}+1}} + \| v \|_{H^{s_{2}+1}} + \| \ti{u} \|_{H^{s_{2}+1}} + \| \ti{v} \|_{H^{s_{2}+1}} \leq C. \nonumber
\end{align}
Applying the Gronwall inequality to (\ref{eq:z3}), we have (\ref{eq:z}).
\end{proof}

\begin{lem}\label{lem:Lip}
Let $X = - \{ K(v^{2})_{x}/2 - v_{x} Kv - v Kv_{x} \} $. Under the same assumption of Proposition \ref{prop:2order} , there exists $C = C(\|u_{0}\|_{H^{s}}, C_{1}', s, s_{2}, T)>0$ such that
\begin{equation}
\| X(t_{1}) - X(t_{2}) \|_{H^{s_{2}}} \leq C | t_{1} - t_{2} | \label{eq:X}
\end{equation}
for all $t_{1}$, $t_{2} \in [0, t_{n}]$.
\end{lem}

\begin{rem}\label{rem:FXG}
Since $F=v_{t} + v v_{x}$, $G=v_{\ta} - Kv$, and $v$ satisfies (\ref{eq:SB}) and (\ref{eq:SL}), we have
\begin{align}
F_{\ta} - KF &= X \, , \, (t,\ta) \in \La_{1}^{(n)}, \nonumber \\
\{ G_{t} + (vG)_{x} \} &= -X  \, , \, (t,\ta) \in \La_{2}^{(n)}. \nonumber
\end{align}
\end{rem}
Before proving Lemma \ref{lem:Lip}, we estimate $v_{t}$ and $v_{\ta}$.
\begin{lem}\label{lem:vt-vta}
Let $j = 1, \, 2, \, 3$ and $3/2 < \si \leq s - j \max \{ 1,p \}$ for each $j$. Under the same assumption of Proposition \ref{prop:2order}, there exists $C = C(\|u_{0}\|_{H^{s}}, C_{1}', s, s_{2}, T)>0$ such that
\begin{align}
\| (\partial_{t})^{j} v(t,\ta) \|_{H^{\si}} &\leq \prod_{m = 0}^{j} \sup_{(t,\ta) \in \Pi_{S}^{(n)}} \| v(t,\ta) \|_{H^{\si+m}}, \label{eq:v(t)j} \\
\| (\partial_{\ta})^{j} v(t,\ta) \|_{H^{\si}} &\leq C \sup_{(t,\ta) \in \Pi_{S}^{(n)} }\| v(t,\ta) \|_{H^{\si + j p}}. \label{eq:v(ta)j}
\end{align}
\end{lem}
\begin{proof}
First, we prove \eqref{eq:v(t)j} for the case $(t,\ta) \in \La_{1}^{(n)}$ with $j = 1$. Since \eqref{K1} and \eqref{eq:SL}, it follows that for $(t,\ta) \in \La_{1}^{(n)} \cap \Sigma_{S,2}^{(n)}$ and $3/2 < \si \leq s$,
\begin{equation}
\partial_{\ta} \| v_{t}(t,\ta) \|_{H^{\si}}^{2} = 2 \langle v_{t}, K v_{t} \rangle_{H^{\si}} \leq 0.
\end{equation}
By \eqref{eq:SB}, it follows that for $(t,\ta) \in \La_{1}^{(n)}$, $l \in \N$ such that $t \in [t_{l-1},t_{l-1/2}]$, and $3/2 < \si \leq s-1$,
\begin{align}
\| v_{t}(t,\ta) \|_{H^{\si}} \leq \| v_{t}(t,t_{l-1}) \|_{H^{\si}} \leq \| v(t,t_{l-1}) \|_{H^{\si}} \| v(t,t_{l-1}) \|_{H^{\si+1}} 
\leq \prod_{j = 0}^{1} \sup_{(t,\ta) \in \Pi_{S}^{(n)}} \| v(t,\ta) \|_{H^{\si+j}}. \nonumber
\end{align}
Next, we prove \eqref{eq:v(t)j} for the case $(t,\ta) \in \La_{1}^{(n)}$ with $j = 2, \, 3$. By induction argument, we have that $\partial_{t} (v v_{x})$ becomes the ($p+1$)-th polynomial with $p$ derivatives. By \eqref{K1}, we have $\partial_{\ta} \| (\partial_{t})^{j}v (t,\ta) \|_{H^{\si}}^{2} \leq 0$. Therefore, the same argument as the case $j = 1$ works and we have \eqref{eq:v(t)j}. 

Next, we prove \eqref{eq:v(t)j} for the case $(t,\ta) \in \La_{2}^{(n)}$ with $j = 1$.
By \eqref{eq:SL}, it follows that for $(t,\ta) \in \La_{2}^{(n)}$ and $3/2 < \si \leq s - 1$, 
\begin{align}
\| v_{t}(t,\ta) \|_{H^{\si}} \leq \| v(t,\ta) \|_{H^{\si}} \| v(t,\ta) \|_{H^{\si+1}} \leq \prod_{j = 0}^{1} \sup_{(t,\ta) \in \Pi_{S}^{(n)}} \| v(t,\ta) \|_{H^{\si+j}}. \nonumber
\end{align}
Then, we have \eqref{eq:v(t)j}. \eqref{eq:v(t)j} for the case $(t,\ta) \in \La_{2}^{(n)}$ with $j = 2, \, 3$ is easily proved since $\partial_{t} (v v_{x})$ becomes the ($p+1$)-th polynomial with $p$ derivatives.

Next, we prove \eqref{eq:v(ta)j} for the case $(t,\ta) \in \La_{2}^{(n)}$  with $j = 1$. By \eqref{eq:SB} and Proposition \ref{prop:ibp} (A), it follows that for $(t,\ta) \in \La_{2}^{(n)} \cap \Sigma_{S,1}^{(n)}$ and $3/2 < \si \leq s-1$,
\begin{equation}
\partial_{t} \| v_{\ta} (t,\ta) \|_{H^{\si}}^{2} = 2 \langle v_{\ta}, (v v_{x})_{\ta} \rangle_{H^{\si}} \leq C \| v_{\ta} \|_{H^{\si}}^{2} \| v \|_{H^{\si+1}}. \nonumber
\end{equation}
By the Gronwall inequality, \eqref{K1}, and \eqref{eq:SL}, it follows that for $(t,\ta) \in \La_{2}^{(n)}$, $l \in \N$ such that $\ta \in [t_{l-1/2},t_{l}]$, and $3/2 < \si \leq s- \max \{ 1,p \}$,
\begin{align}
\| v_{\ta}(t,\ta) \|_{H^{\si}} &\leq e^{C C_{1}' (t-t_{l - 1/2})} \| v_{\ta} (t_{l - 1/2},\ta) \|_{H^{\si}} \nonumber \\
&\leq e^{C C_{1}' \Delta t} \| v (t_{l - 1/2},\ta) \|_{H^{\si + p}} \leq e^{C C_{1}' \Delta t} \sup_{(t,\ta) \in \Pi_{S}^{(n)} }\| v(t,\ta) \|_{H^{\si + p}}. \nonumber
\end{align}
Next, we prove \eqref{eq:v(ta)j} for the case $(t,\ta) \in \La_{2}^{(n)}$  with $j = 2, \, 3$. In the same manner as the case $j = 1$, by the Gronwall inequality, we have $\| (\partial_{\ta})^{j} v (t,\ta) \|_{H^{\si}} \leq e^{C C_{1}' (t-t_{l - 1/2})} \| (\partial_{\ta})^{j} v (t_{l - 1/2},\ta) \|_{H^{\si}}$. Thus, by \eqref{eq:SL}, we have \eqref{eq:v(ta)j}. Finally, we prove \eqref{eq:v(ta)j} for the case $(t,\ta) \in \La_{1}^{(n)}$ with $j = 1, \, 2, \, 3$. By \eqref{K1} and \eqref{eq:SL}, it follows that for $(t,\ta) \in \La_{1}^{(n)}$ and $3/2 < \si \leq s-j$,
\begin{align}
\| (\partial_{\ta})^{j} v (t,\ta) \|_{H^{\si}} \leq \| v (t,\ta) \|_{H^{\si + j p}} \leq  \sup_{(t,\ta) \in \Pi_{S}^{(n)} }\| v(t,\ta) \|_{H^{\si + j p}}. \nonumber
\end{align}
Therefore, we have \eqref{eq:v(ta)j}. 
\end{proof}
Next, we prove Lemma \ref{lem:Lip}.
\begin{proof} 
For the proof of Lemma \ref{lem:Lip}, we only need to prove the boundness of $X_{t}$ and $X_{\ta}$ in $H^{s_{2}}$. We have the boundness for $(t,\ta) \in \La_{2}^{(n)}$ in the same manner as for $(t,\ta) \in \La_{1}^{(n)}$, so we only give the proof for the case $(t,\ta) \in \La_{1}^{(n)}$. First, we prove the boundness of $X_{t}$ in $H^{s_{2}}$. By the definition of $X$, we have
\begin{align}
X_{t} = (Kv_{t}) v_{x} + (Kv_{x}) v_{t} + v_{x t} (Kv) + (Kv_{xt}) v - K (v_{xt} v) - K (v_{t} v_{x}). \nonumber
\end{align}
Since $s_{2} + 2 + p \leq s$, we have
\begin{align}
\| X_{t} \|_{H^{s_{2}}} &\leq C \| v \|_{H^{s_{2} + 2 + p}} \| v \|_{H^{s_{2} + 1 + p}} \leq C.  \label{eq:Xt}
\end{align}
Next, we prove for the boundness of $X_{\ta}$ in $H^{s_{2}}$. Since $s_{2} + 1 + 2p \leq s$, \eqref{eq:SL}, and Remark \ref{rem:FXG}, in the same manner as \eqref{eq:Xt}, we have 
\begin{align}
\| X_{\ta} \|_{H^{s_{2}}} &\leq \|(K^{2} v) v_{x} \|_{H^{s_{2}}} + 2 \|(Kv) (Kv_{x}) \|_{H^{s_{2}}} + \| v (K^{2} v_{x}) \|_{H^{s_{2}}} \nonumber \\
&\quad + \| K(v (Kv_{x})) \|_{H^{s_{2}}} + \| K((Kv) v_{x}) \|_{H^{s_{2}}} \nonumber \\
&\leq C. \nonumber
\end{align} 
\end{proof}
Next, we estimate $H + \ti{H} = F + G + \ti{F} + \ti{G}$. In view of $\ti{F} = G = 0$ in $\La_{1}^{(n)}$, $F = \ti{G} = 0$ in $\La_{2}^{(n-1)}$, and Remark \ref{rem:FXG}, it is natural to estimate $\ti{F} + G$ and $F + \ti{G}$. 
\begin{lem}\label{lem:FG}
Let $l \in \N$, $F$ satisfy \eqref{eq:ftf}, and $G$ satisfy \eqref{eq:ftg}. Under the same assumptions of Proposition \ref{prop:2order}, there exists $C = C (\|u_{0}\|_{H^{s}}, C_{1}', s, s_{2}, T)>0$ such that
\begin{equation}
\Big{\|} F(t) + G (t + \frac{\Delta t}{2}) \Big{\|}_{H^{s_{2}}} \leq C (\Delta t)^{2} \nonumber
\end{equation}
for all $t \in [t_{l-1},t_{l- 1/2}] \subset [0, t_{n-1/2}]$, and
\begin{equation}
\Big{\|} F(t) + G (t - \frac{\Delta t}{2}) \Big{\|}_{H^{s_{2}}} \leq C (\Delta t)^{2} \nonumber
\end{equation}
for all $t \in [t_{l-1},t_{l- 1/2}] \subset [t_{1}, t_{n-1/2}]$.
\end{lem}

\begin{proof}
Let $\Phi = F_{tt} + 2 F_{t \ta} + F_{\ta \ta}$ and $\Psi = G_{tt} + 2 G_{t \ta} + G_{\ta \ta}$. By Taylor expansion at $t = t_{l-1}$, we have
\begin{align}
&F(t) + G(t \pm \frac{\Delta t}{2}) \nonumber \\
&\quad = \{ F(t_{l-1}) + G(t_{l -1 \pm \frac{1}{2}}) \} + (t - t_{l-1}) \{ (F_{t} + F_{\ta}) (t_{l-1}) + (G_{t} + G_{\ta}) (t_{l -1 \pm \frac{1}{2}})\} \nonumber \\
&\qquad + \frac{(t - t_{l-1})^{2}}{2} \int_{0}^{1} \{ \Phi(\theta(t - t_{l-1}) + t_{l-1}) + \Psi (\theta(t \pm \frac{\Delta t}{2} - t_{l-1\pm \frac{1}{2}}) + t_{l-1\pm \frac{1}{2}}) \} d \theta. \nonumber
\end{align}
Since $F = 0$ in $\La_{2}^{(n)} \cup \Sigma_{S,1}^{(n)}$ because of \eqref{eq:SB} and \eqref{eq:ftf}, and $G = 0$ in $\La_{1}^{(n)} \cup \Sigma_{S,2}^{(n)}$ because of \eqref{eq:SL} and \eqref{eq:ftg}, we have $F(t_{l-1}) = G(t_{l-1 \pm 1/2}) = 0$ and $F_{t}(t_{l-1}) = G_{\ta}(t_{l-1 \pm 1/2}) = 0$. By Remark \ref{rem:FXG} and $F(t_{l-1}) = G(t_{l-1 \pm 1/2}) = 0$, we have $F_{\ta}(t_{l-1}) + G_{t}(t_{l-1 \pm \frac{1}{2}}) = X(t_{l-1}) - X(t_{l -1  \pm \frac{1}{2}})$. Then, we have the second order estimate in $\Delta t$ of the second term in $H^{s_{2}}$ space by Lemma \ref{lem:Lip}. $\| \Phi \|_{H^{s_{2}}} + \| \Psi \|_{H^{s_{2}}} \leq C$ is proved by \eqref{eq:SB}--\eqref{IVS}, Proposition \ref{prop:ibp}, \eqref{eq:LHs}, (\ref{eq:v(t)j}) , and (\ref{eq:v(ta)j}). Therefore it follows that for $t \in [t_{l-1},t_{l- 1/2}]$,
\begin{align}
\| F(t) + G(t \pm \frac{\Delta t}{2}) \|_{H^{s_{2}}} &\leq |t - t_{l-1}| \| X(t_{l-1 \pm \frac{1}{2}}) - X(t_{l -1}) \|_{H^{s_{2}}} + \frac{C}{2} |t - t_{l-1}|^{2} \nonumber \\
&\leq C (\Delta t)^{2}.
\end{align}
\end{proof}
\begin{lem}\label{lem:small}
Let $s_{2} = s - 3 \max \{1,p \} > 3/2$. Assume that $u \in C([0,t_{1/2}]: H^{s})$ satisfies $(\ref{eq:DBO})$ on $[0,t_{1/2}]$ and $v \in C(\La_{1}^{(1)} : H^{s})$ satisfies $(\ref{eq:SB})$--$(\ref{IVS})$ on $\La_{1}^{(1)}$. Assume that a constant $C_{1}' > 0$ satisfies $\sup_{(t,\tau) \in \La_{1}^{(1)}} \| v(t,\tau) \|_{H^{s}} \leq C_{1}'$. Then, there exists $C = C (\sup_{[0, t_{1/2}]} \|u(t)\|_{H^{s}}, C_{1}', s, s_{2})>0$ such that
\begin{equation}
\| w(t) \|_{H^{s_{2}}} \leq C (\Delta t)^{2} \nonumber
\end{equation}
for all $t \in [0, t_{1/2}]$.
\end{lem}
\begin{proof}
We use the Taylor expantion of $w(t)$ at $t = 0$. That is
\begin{align}
w(t) &= w(0) + t \big{(} \partial_{t} v(t,\ta) |_{t = \ta = 0} + \partial_{\ta} v(t,\ta) |_{t = \ta = 0} - u'(0) \big{)} \nonumber \\
&\quad + \frac{t^{2}}{2} \int_{0}^{1} \{ \partial_{tt} v(\si t,\si t) + 2 \partial_{t \ta} v(\si t,\si t) + \partial_{\ta \ta} v(\si t,\si t) \} d \si. \nonumber
\end{align}
The term of the 0th order of $t$ is $0$ because of $w(0)=v(0,0) - u(0)$ and $u(0) = v(0,0)$. Applying $(\ref{eq:DBO})$, $(\ref{eq:SB})$, $(\ref{eq:SL})$, and $u(0) = v(0,0)$ to the term of the 1st order of $t$, we have
\begin{align}
\partial_{t} v(t,\ta) |_{t = \ta = 0} &+ \partial_{\ta} v(t,\ta) |_{t = \ta = 0} - u'(0) \nonumber \\
&= -v(0,0) v_{x}(0,0) + K v(0,0) + u(0)u_{x}(0) - K u(0) = 0. \nonumber
\end{align}
By \eqref{eq:v(t)j} and \eqref{eq:SL}, it follows that
\begin{align}
&\| \partial_{tt} v(\si t,\si t) + 2 \partial_{t \ta} v(\si t,\si t) + \partial_{\ta \ta} v(\si t,\si t) \|_{H^{s_{2}}} \nonumber \\
&\quad \leq \| \partial_{tt} v(\si t,\si t) \|_{H^{s_{2}}} + 2 \| K \partial_{t} v(\si t,\si t) \|_{H^{s_{2}}} + \| K^{2} v(\si t,\si t) \|_{H^{s_{2}}} \nonumber \\
&\quad \leq \prod_{j = 0}^{2} \max_{(t,\ta) \in \La_{1}^{(1)}} \| v(t,\ta) \|_{H^{s_{2}+j}} + \prod_{j = 0}^{1} \max_{(t,\ta) \in \La_{1}^{(1)}} \| v(t,\ta) \|_{H^{s_{2} + p +j}} + \max_{(t,\ta) \in \La_{1}^{(1)}} \| v(t,\ta) \|_{H^{s_{2} + 2p}} \nonumber \\
&\quad \leq C
\end{align}
Therefore, $\| w(t) \|_{H^{s_{2}}} \leq C t^{2}$. Since $t \leq (\Delta t)/2$, we have the desired result. 
\end{proof}

Next, we prove (\ref{eq:ave}) by Lemmas \ref{lem:GWS}, \ref{lem:small}, Theorem \ref{thm:L}, and Remark \ref{rem:rev}.
\begin{proof}
First, we prove three inequalities
\begin{equation}
\| z(0) \|_{H^{s_{2}}} \leq C (\Delta t)^{2} , \, \| \ti{u} - u \|_{H^{s_{2} + 1}} \leq C \Delta t , \,  \| w \|_{H^{s_{2} + 1}} \leq C \Delta t. \label{eq:3ests}
\end{equation}
We have the first inequality in $(\ref{eq:3ests})$ by applying Lemma \ref{lem:small} for $t = (\Delta t)/2$. Note that $z(0) = w(0) + \ti{w}(0) = w(\frac{\Delta t}{2})$. To prove the second one, we use $(\ref{eq:DBO})$, $(\ref{bdd-u})$, and $\tilde{u}(t) - u(t) = \int_{0}^{\frac{\Delta t}{2}} u'(t + \si) d \si$. The third one is already proved as Theorem \ref{thm:L} and Remark \ref{rem:rev}.  

Finally, we prove \eqref{eq:ave}. For $(t,t) \in \La_{1}^{(n)}$, since $\ti{F} = G = 0$, we have $\| H + \ti{H} \|_{H^{s_{2}}} = \| F + \ti{G} \|_{H^{s_{2}}} \leq C (\Delta t)^{2}$ by Lemma \ref{lem:FG}. For $(t,t) \in \La_{2}^{(n-1)}$ (not $\La_{2}^{(n)}$), since $F = \ti{G} = 0$, we have $\| H + \ti{H} \|_{H^{s_{2}}} = \| \ti{F} + G \|_{H^{s_{2}}} \leq C (\Delta t)^{2}$ by Lemma \ref{lem:FG}. 
Thus, it follows that for $(t,t) \in \La_{1}^{(n)} \cup \La_{2}^{(n-1)}$, 
\begin{align}
\| H(t) + \ti{H}(t) \|_{H^{s_{2}}} \leq C (\Delta t)^{2}. \label{eq:HtiH}
\end{align}
Since $(t,t) \in \La_{1}^{(n)} \cup \La_{2}^{(n-1)}$ is equivalent to $t \in [0, t_{n-1/2}]$, \eqref{eq:ave} follows by applying \eqref{eq:3ests} and \eqref{eq:HtiH} to Lemma \ref{lem:GWS}.
\end{proof}

Next, we prove \eqref{eq:dif}.
We use the following lemma to prove \eqref{eq:dif} later.

\begin{lem}\label{lem:dif2}
Let $l \in \N$ and assumptions of Proposition \ref{prop:2order} hold.\\
$(A)$ Then there exists $C = C (C_{0}, C_{1}', s, s_{2}, T)>0$ such that for all $t \in [t_{l-1}, t_{l - 1/2}] \subset [0, t_{n}]$,
\begin{equation*}
\| w(t) - w(t_{l-1}) \|_{H^{s_{2}}} \leq C (\Delta t)^{2}.
\end{equation*}
$(B)$ Then there exists $C = C (C_{0}, C_{1}', s, s_{2}, T)>0$ such that for all $t \in [t_{l - 1/2}, t_{l}] \subset [0, t_{n}]$,
\begin{equation*}
\| w(t) - w(t_{l - 1/2}) \|_{H^{s_{2}}} \leq C (\Delta t)^{2}.
\end{equation*}
\end{lem}
\begin{proof}
We have (B) in the same manner as (A), so we only prove (A).
Let $V(t,x)$ satisfy
\begin{equation}
\begin{cases}
\partial_{t} V - V V_{x} + K V = 0, \hs{15pt}  (x,t) \in \R \times [t_{l-1},T],\\
V (\cdot, t_{l-1}) = v(\cdot,t_{l-1},t_{l-1}) \in H^{s}(\R).
\end{cases}
\label{eq:DBOv}
\end{equation}
Note that the first equations in \eqref{eq:DBO} and \eqref{eq:DBOv} are the same, and we may also apply \eqref{bdd-u} to \eqref{eq:DBOv}. 
First, we decompose $ w(t) - w(t_{l-1})$ as
\begin{align}
&w(t) - w(t_{l-1}) \nonumber \\
&\quad = \{ v (t) - V(t) \} + \{ V(t) - v(t_{l-1},t_{l-1}) \} - \{ u(t) - u(t_{l-1}) \}.
\label{eq:tri2}
\end{align}
We have the second-order approximation in $\Delta t$ from the first term in \eqref{eq:tri2} by Lemma \ref{lem:small}.
The second and the third terms are
\begin{align}
&\{ V(t) - v(t_{l-1},t_{l-1}) \} - \{ u(t) - u(t_{l-1}) \} \nonumber \\
&\quad = \int_{t_{l-1}}^{t} \{ V_{t}(t') - u_{t}(t') \} d t' \nonumber \\
&\quad = \int_{t_{l-1}}^{t} \{ V(t') V_{x}(t') - K V(t') - u(t') u_{x}(t') + K u(t') \} d t' \nonumber \\
&\quad = \int_{t_{l-1}}^{t} \{ V(t')(V(t') - u(t'))_{x} + (V(t') - u(t')) u_{x}(t') - K (V - u)(t') \} d t'. \nonumber
\end{align}
Let $s_{2}' = s_{2}+ \max \{  1,p \} (= s - 2 \max \{  1,p \})$. Since $\| V \|_{H^{s_{2}}} \leq C$ and $\| u_{x} \|_{H^{s_{2}}} \leq C_{0}$, we have
\begin{align}
&\Big{\|} \{ V(t) - v(t_{l-1},t_{l-1}) \} - \{ u(t) - u(t_{l-1}) \} \Big{\|}_{H^{s_{2}}} \nonumber \\
&\quad \leq C \int_{t_{l-1}}^{t} \| V(t') - u(t') \|_{H^{s_{2}'}} d t' \leq C \Delta t \sup_{t \in [t_{l-1},t_{l-1/2}]} \| V(t) - u(t) \|_{H^{s_{2}'}}. \nonumber
\end{align}
$\| V(t) - u(t) \|_{H^{s_{2}'}}$ is estimated from \eqref{eq:DBO}, \eqref{bdd-u}, \eqref{eq:DBOv}, Theorem \ref{thm:L}, and Remark \ref{rem:rev} as
\begin{align}
&\| V(t) - u(t) \|_{H^{s_{2}'}} \nonumber \\
&\quad \leq \| V(t) - v(t_{l-1},t_{l-1} ) \|_{H^{s_{2}'}} + \| w(t_{l-1}) \|_{H^{s_{2}'}} + \| u(t_{l-1}) - u (t) \|_{H^{s_{2}'}} \nonumber \\
&\quad \leq \int_{t_{l-1}}^{t} \| V'(t') \|_{H^{s_{2}'}} d t' + C \Delta t + \int_{t_{l-1}}^{t} \| u'(t') \|_{H^{s_{2}'}} d t' \leq C \Delta t.
\label{eq:tri3}
\end{align}
Thus, we have the second-order approximation in $\Delta t$ from the second and the third terms in \eqref{eq:tri2}. Therefore we have Lemma \ref{lem:dif2} (A).
\end{proof}

Finally, we prove (\ref{eq:dif}).
\begin{proof}
Let $l \in \N$. By Lemma \ref{lem:dif2}, it follows that for $t \in [t_{l-1} , t_{l - 1/2}] \subset [0, t_{n-1/2}]$,
\begin{align}
&\| \ti{w}(t) - w(t) \|_{H^{s_{2}}} \nonumber \\
&\leq \| w(t + \Delta t /2) - w(t_{l - 1/2}) \|_{H^{s_{2}}} + \| w(t_{l - 1/2}) - w(t_{l-1}) \|_{H^{s_{2}}} + \| w(t_{l-1}) - w(t) \|_{H^{s_{2}}} \nonumber \\
&\leq C (\Delta t)^{2}, \nonumber
\end{align}
and for $t \in [t_{l - 1/2} , t_{l}] \subset [0, t_{n-1/2}]$,
\begin{align}
&\| \ti{w}(t) - w(t) \|_{H^{s_{2}}} \nonumber \\
&\leq \| w(t + \Delta t /2) - w(t_{l}) \|_{H^{s_{2}}} + \| w(t_{l}) - w(t_{l - 1/2}) \|_{H^{s_{2}}} + \| w(t_{l - 1/2}) - w(t) \|_{H^{s_{2}}} \nonumber \\
&\leq C (\Delta t)^{2}. \nonumber
\end{align}
\end{proof}
\subsection*{Acknowledgements}
The author would like to thank his advisor, Kotaro Tsugawa, for suggesting the problem and for all the guidance and encourarement along the way

\end{section}

\end{document}